\documentclass[11pt]{amsart}
\usepackage{cases}
\usepackage{amsmath}
\usepackage{amscd}
\usepackage{amssymb}
\usepackage{amsfonts}
\newtheorem{theorem}{Theorem}[section]
\newtheorem{lemma}[theorem]{Lemma}
\newtheorem{corollary}[theorem]{Corollary}

\theoremstyle{definition}

\newtheorem{proposition}[theorem]{Proposition}

\theoremstyle{remark}
\newtheorem{remark}[theorem]{Remark}
\numberwithin{equation}{section}

\begin{document}
\title[Elliptic gradient estimates for a nonlinear heat equation]
{Elliptic gradient estimates for a nonlinear heat equation and applications}

\author{Jia-Yong Wu}

\address{Department of Mathematics, Shanghai Maritime University,
1550 Haigang Avenue, Shanghai 201306, P. R. China}
\email{jywu81@yahoo.com}

\date{\today}


\thanks{This work is partially supported by NSFC (11101267, 11271132).}

\begin{abstract}
In this paper, we study elliptic gradient estimates for a nonlinear $f$-heat
equation, which is related to the gradient Ricci soliton and the weighted
log-Sobolev constant of smooth metric measure spaces. Precisely, we obtain
Hamilton's and Souplet-Zhang's gradient estimates for positive solutions
to the nonlinear $f$-heat equation only assuming the Bakry-\'Emery
Ricci tensor is bounded below. As applications, we prove parabolic Liouville
properties for some kind of ancient solutions to the nonlinear $f$-heat equation.
Some special cases are also discussed.
\end{abstract}
\maketitle

\section{Introduction}\label{Int1}

\subsection{Background}
This is a sequel to our previous work \cite{[Wu2]}. In that paper we
proved elliptic gradient estimates for positive solutions to the
$f$-heat equation on smooth metric measure spaces with only the
Bakry-\'Emery Ricci tensor bounded below. We also applied the
results to get parabolic Liouville theorems for some ancient
solutions to the $f$-heat equation. In this paper we will investigate
elliptic gradient estimates and Liouville properties for positive
solutions to a nonlinear $f$-heat equation (see equation \eqref{equ1} below)
on complete smooth metric measure spaces.

Recall that an $n$-dimensional smooth metric measure space $(M^n,g,e^{-f}dv)$
is a complete Riemannian manifold $(M^n,g)$ endowed with a weighted measure
$e^{-f}dv$ for some $f\in C^\infty(M)$, where $dv$ is the volume element of
the metric $g$. The associated $m$-Bakry-\'Emery Ricci tensor \cite{[BE]}
is defined by
\[
Ric_f^m:=Ric+\nabla^2 f-\frac{1}{m}df\otimes df
\]
for some constant $m>0$, where $Ric$ and $\nabla^2$ denote the Ricci tensor
and the Hessian of the metric $g$. When $m=\infty$, we have the ($\infty$-)Bakry-\'Emery
Ricci tensor
\[
Ric_f:=Ric+\nabla^2 f.
\]
The Bochner formula for $Ric_f^m$ can be read as (see also \cite{[Wu2]})
\begin{equation}\label{weighBoch}
\begin{split}
\frac 12\Delta_f|\nabla u|^2&=|\nabla^2
u|^2+\langle\nabla\Delta_f u, \nabla u\rangle+Ric_f(\nabla u, \nabla u)\\
&\geq\frac{(\Delta_f u)^2}{m+n}
+\langle\nabla\Delta_f u, \nabla u\rangle+Ric_f^m(\nabla u, \nabla u)
\end{split}
\end{equation}
for any $u\in C^\infty(M)$. When $m<\infty$, \eqref{weighBoch} could be
viewed as the Bochner formula for the Ricci tensor of an $(n+m)$-dimensional
manifold. Hence many geometric and topological properties for manifolds
with Ricci tensor bounded below can be possibly extended to smooth metric
measure spaces with $m$-Bakry-\'Emery Ricci tensor bounded below, see for
example \cite{[LD],[Lott]}. When $m=\infty$, the ($\infty$-)Bakry-\'Emery
Ricci tensor is related to the gradient Ricci soliton
\[
Ric_f=\lambda\, g
\]
for some constant $\lambda$, which plays an important role in Hamilton's
Ricci flow as it corresponds to the self-similar solution and arises as
limits of dilations of singularities in the Ricci flow \cite{[Hami]}.
A Ricci soliton is said to be shrinking, steady, or expanding according
to $\lambda>0$, $\lambda=0$ or $\lambda<0$. On the gradient estimate,
the smooth function $f$ is often called a potential function. We refer
\cite{[Cao1]} and the references therein for further discussions.

On smooth metric measure space $(M,g,e^{-f}dv)$, the $f$-Laplacian $\Delta_f$
is defined by
\[
\Delta_f:=\Delta-\nabla f\cdot\nabla,
\]
which is self-adjoint with respect to the weighted measure. The associated
$f$-heat equation is defined by
\begin{equation}\label{weiheat}
\frac{\partial u}{\partial t}=\Delta_f \,u
\end{equation}
If $u$ is independent of time $t$, then it is $f$-harmonic function. In the
past few years, various Liouville properties for $f$-harmonic functions
were obtained, see for example \cite{[Bri]}, \cite{[LD]}, \cite{[LD2]}, \cite{[MuWa]},
\cite{[WW]}, \cite{[Wu]}, \cite{[WuWu0]}, \cite{[WuWu]}, and the references
therein. Recently, the author \cite{[Wu2]} proved elliptic gradient estimates
and parabolic Liouville properties for $f$-heat equation under some assumptions
of ($\infty$-)Bakry-\'Emery Ricci tensor.

In this paper, we will study analytical and geometrical properties for positive
solutions to the equation
\begin{equation}\label{equ1}
\frac{\partial u}{\partial t}=\Delta_f\, u+au\ln u,
\end{equation}
where $a\in\mathbb{R}$, on complete smooth metric measure spaces $(M,g,e^{-f}dv)$
with only the Bakry-\'Emery Ricci tensor bounded below. Here we assume
$M$ has no boundary. It is well-known that all solutions to its Cauchy problem
exist for all time. Under the assumption of $Ric_f$, we shall prove
local elliptic (Hamilton's type and Souplet-Zhang's type) gradient estimates for
positive solutions to the nonlinear $f$-heat equation \eqref{equ1}. As applications,
we prove parabolic Liouville properties for the nonlinear $f$-heat equation \eqref{equ1}.

Historically, gradient estimates for the harmonic function on manifolds
were discovered by Yau \cite{[Yau]} and Cheng-Yau \cite{[Cheng-Yau]} in 1970s.
It was extended to the so-called Li-Yau gradient estimate for the heat
equation by Li and Yau \cite{[Li-Yau]} in 1980s. In 1990s, Hamilton \cite{[Ham93]}
gave an elliptic type gradient estimate for the heat equation on closed
manifolds, which was later generalized to the non-compact case by Kotschwar
\cite{[Kots]}. In 2006, Souplet and Zhang \cite{[Sou-Zh]} proved a localized
Cheng-Yau type estimate for the heat equation by adding a logarithmic correction 
term. Integrating Hamilton's or Souplet-Zhang's gradient estimates along 
space-time paths, their estimates exhibit an interesting phenomenon
that one can compare the temperature of two different points at the same time
provided the temperature is bounded. However, Li-Yau gradient estimate only
provides the comparison at different times.

Equation \eqref{equ1} has some relations to the geometrical quantities. On one
hand, the time-independent version of \eqref{equ1} with constant function $f$
is linked with gradient Ricci solitons, for example, see \cite{[Ma],[Yang]}
for detailed explanations. On the other hand, the steady-state version of
\eqref{equ1} is closely related to weighted log-Sobolev constants of smooth
metric measure spaces (Riemmann manifolds case due to Chung-Yau
\cite{[Chu-Yau]}). Recall that, weighted log-Sobolev constants $S_M$,
associated to a closed smooth metric measure space $(M^n,g,e^{-f}dv)$,
are the smallest positive constants such that the weighted logarithmic-Sobolev
inequality
\[
\int_M u^2\ln(u^2)\,e^{-f}dv\leq S_M\int_M |\nabla u|^2\,e^{-f}dv
\]
holds for all smooth function $u$ on $M$ satisfying $\int_M u^2 e^{-f}dv=V_f(M)$.
In particular, the case of Euclidean space $\mathbb{R}^n$ equipped with the
Gaussian measure
\[
e^{-f}dv:=(4\pi)^{-\frac n2}\exp\left(-\frac{|x|^2}{4}\right)dx
\]
is inequivalent to the original log-Sobolev inequality due to L. Gross \cite{[Gross]}.

If function $u$ achieves the weighted log-Sobolev constant and satisfies
$\int_M u^2 e^{-f}dv=V_f(M)$, that is,
\[
S_M=\frac{\int_M|\nabla u|^2\,e^{-f}dv}{\int_M u^2\ln u^2\,e^{-f}dv}
=\inf_{\phi\not=0}\frac{\int_M|\nabla\phi|^2\,e^{-f}dv}{\int_M\phi^2\ln\phi^2\,e^{-f}dv}.
\]
Using the Lagrange's method with respect to weighted measure $e^{-f}dv$,
we have
\[
\frac{-2\Delta_f u}{\int_Mu^2\ln u^2\,e^{-f}dv}
-\frac{\int_M|\nabla u|^2\,e^{-f}dv}{(\int_Mu^2\ln u^2\,e^{-f}dv)^2}\left(2u\ln u^2+2u\right)
+c_1u=0
\]
for some constant $c_1$. By the definition of $S_M$, this can be reduced to
\begin{equation}\label{vary}
-\Delta_f u-S_M(u\ln u^2+u)+c_2u=0
\end{equation}
for the constant
\[
c_2=\frac{c_1}{2}\int_Mu^2\ln u^2\,e^{-f}dv.
\]
Notice that multiplying \eqref{vary} by $u$ and integrating it with respect to
the weighted measure $e^{-f}dv$, we have
\[
\int_M |\nabla u|^2\,e^{-f}dv-S_M\int_M u^2(\ln u^2+1)\,e^{-f}dv
+c_2\int_Mu^2\,e^{-f}dv=0,
\]
which implies $S_M=c_2$. Therefore \eqref{vary} can be simplified as
\begin{equation}\label{sobl}
\Delta_f\, u+S_M\, u\ln u^2=0,
\end{equation}
which is an elliptic version of \eqref{equ1}. For \eqref{sobl},
if $Ric^m_f\geq 0$, using \eqref{weighBoch} instead of the classical
Bochner formula for the Ricci tensor and following Chung-Yau's
arguments \cite{[Chu-Yau]}, we immediately get
\[
\sup u\leq e^{(n+m)/2},\quad\quad |\nabla\ln u|^2+S_M\ln u^2\leq(n+m)S_M
\]
and
\[
S_M \geq\min\left\{\frac{\lambda_1}{8e}, \frac{1}{(n+m)d^2}\right\},
\]
where $\lambda_1$ and $d$ denote the first nonzero eigenvalue of the
$f$-Laplacian and the diameter of $(M^n,g,e^{-f}dv)$.

\subsection{Main results}
Our first result gives a local Hamilton's gradient estimate for any positive
solution to the equation \eqref{equ1}.
\begin{theorem}\label{main}
Let $(M,g,e^{-f}dv)$ be an $n$-dimensional complete smooth metric measure
space. For any point $x_0\in M$ and $R\geq2$, $Ric_f\geq-(n-1)K$ for some
constant $K\geq0$ in $B(x_0,R)$. Let $0<u(x,t)\leq D$ for some constant
$D$, be a smooth solution to the equation \eqref{equ1} in
$Q_{R,T}\equiv B(x_0,R)\times[t_0-T,t_0]\subset M\times(-\infty,\infty)$,
where $t_0\in \mathbb{R}$ and $T>0$.
\begin{enumerate}
\item [(i)] If $a\geq0$, then there exists a constant $c(n)$ such that
\begin{equation}\label{heor1}
\frac{|\nabla u|}{\sqrt{u}}\leq c(n) \sqrt{D}\left(\frac{1}{R}+\sqrt{\frac{|\alpha|}{R}}+\frac{1}{\sqrt{t{-}t_0{+}T}}+\sqrt{K}{+}\sqrt{c_1(n,K,a,D)}\right)
\end{equation}
in $Q_{R/2, T}$ with $t\neq t_0-T$, where $c_1(n,K,a,D)=\max\{2(n-1)K+a(2+\ln D),\,0\}$.

\item [(ii)] If $a<0$, further assuming that $\delta\leq u(x,t)\leq D$ for some constant $\delta>0$,
then there exists a constant $c(n)$ such that
\begin{equation}\label{heor2}
\frac{|\nabla u|}{\sqrt{u}}\leq c(n) \sqrt{D}\left(\frac{1}{R}+\sqrt{\frac{|\alpha|}{R}}+\frac{1}{\sqrt{t{-}t_0{+}T}}+\sqrt{K}+\sqrt{c_2(n,K,a,\delta)}\right)
\end{equation}
in $Q_{R/2, T}$ with $t\neq t_0-T$, where $c_2(n,K,a,\delta)=\max\{2(n-1)K+a(2+\ln\delta),\,0\}$.
\end{enumerate}
Here, $\alpha:=\max_{\{x|d(x,x_0)=1\}}\Delta_f\,r(x)$,
where $r(x)$ is the distance function to $x$ from base point $x_0$.
\end{theorem}
\begin{remark}
Cao, Fayyazuddin Ljungberg and Liu \cite{[CFL]} proved Li-Yau type gradient
estimates for equation \eqref{equ1} with constant function $f$; our results
belong to the elliptic type. An distinct feature of Theorem \ref{main} is
that the gradient estimates hold only assuming the ($\infty$-)Bakry-\'Emery
Ricci tensor is bounded below (without any assumption on $f$).

We also remark that our proof is a little different from Yau's original proof
\cite{[Yau]}. In Yau's case, the proof is to compute the evolution of quantity
$\ln u$, then multiply by a cut-off function and apply the maximum principle.
In our case, we compute the evolution of quantity $u^{1/3}$ instead of $\ln u$.
Moreover, our proof not only applies some arguments of Souplet-Zhang
\cite{[Sou-Zh]}, where the maximum principle in a local space-time supported
set is discussed, but also uses some proof tricks of Bailesteanua-Cao-Pulemotov
\cite{[BCP]}, Li \cite{[Lij]} and Wei-Wylie's comparison theorem \cite{[WW]}.
\end{remark}

\vspace{0.5em}

An immediate application of Theorem \ref{main} is the parabolic Liouville
property for the nonlinear $f$-heat equation. Similar results
appeared in \cite{[Jiang]}.
\begin{theorem}\label{app1}
Let $(M,g,e^{-f}dv)$ be an $n$-dimensional complete smooth metric
measure space with $Ric_f\geq 0$.
\begin{enumerate}
\item [(i)] When $a>0$, if $u(x,t)$ is a positive ancient solution to equation
\eqref{equ1} (that is, a solution defined in all space and negative time)
such that $0<u(x,t)\leq e^{-2}$, then $u$ does not exist.

\item [(ii)] When $a<0$, let $u(x,t)$ be a positive ancient solution
to equation \eqref{equ1}.
If $e^{-2}\leq u(x,t)\leq D$ for some constant $D<1$, then $u$ does not exist;
if $e^{-2}\leq u(x,t)\leq D$ for some constant $D\geq 1$, then $u\equiv 1$.

\item [(iii)] When $a=0$, if $u(x,t)$ is a positive ancient solution to equation
\eqref{weiheat} such that $u(x,t)=o\Big(\big[r^{1/2}(x)+|t|^{1/4}\big]^2\Big)$ near
infinity, then $u$ is constant.
\end{enumerate}
\end{theorem}

\vspace{0.5em}

Theorem \ref{app1} immediately implies the following result.
\begin{corollary}\label{app1b}
Let $(M,g,e^{-f}dv)$ be an $n$-dimensional closed smooth metric measure space
with $Ric_f\geq 0$. If positive smooth function $u(x)$ achieves the weighted
log-Sobolev constant $S_M$ and satisfies
\[
\int_M u^2 e^{-f}dv=V_f(M),
\]
then $u(x)>e^{-2}$.
\end{corollary}

\vspace{0.5em}

For more interesting special cases and applications of Theorem \ref{main}, see
Section \ref{sec3} for furthermore discussion.

\

Our second result gives a Souplet-Zhang's elliptic gradient estimate for
positive solutions to the nonlinear $f$-heat equation \eqref{equ1}. The
proof mainly adopts the arguments of Bailesteanua-Cao-Pulemotov \cite{[BCP]},
Souplet-Zhang \cite{[Sou-Zh]} and Brighton \cite{[Bri]} (see also \cite{[Wu2]}).
\begin{theorem}\label{main0}
Let $(M,g,e^{-f}dv)$ be an $n$-dimensional complete smooth metric measure
space. For any point $x_0\in M$ and $R\geq 2$, $Ric_f\geq-(n-1)K$ for some
constant $K\geq0$ in $B(x_0,R)$. Let $0<u(x,t)\leq D$ for some constant
$D$, be a smooth solution to $f$-heat equation \eqref{equ1} in
$Q_{R,T}\equiv B(x_0,R)\times[t_0-T,t_0]\subset M\times(-\infty,\infty)$,
where $t_0\in \mathbb{R}$ and $T>0$.
\begin{enumerate}
\item [(i)] If $a\geq0$, then there exists a constant $c(n)$ such that
\begin{equation}\label{heork1}
\frac{|\nabla u|}{u}\leq c(n)\left(\sqrt{\frac{1{+}|\alpha|}{R}}+\frac{1}{\sqrt{t-t_0+T}}+\sqrt{K}+\sqrt{a(\kappa+1)}\right)
\left(1+\ln \frac Du\right)
\end{equation}
in $Q_{R/2, T}$ with $t\neq t_0-T$, where $\kappa=\max\{|\ln D|, 1\}$.

\item [(ii)] If $a<0$, then there exists a constant $c(n)$ such that
\begin{equation}\label{heork2}
\frac{|\nabla u|}{u}\leq c(n)\left(\sqrt{\frac{1{+}|\alpha|}{R}}{+}\frac{1}{\sqrt{t-t_0+T}}{+}\sqrt{K}{+}\sqrt{c_3(n,a,K)}{+}\sqrt{-a\kappa}\right)
\left(1{+}\ln \frac Du\right)
\end{equation}
in $Q_{R/2, T}$ with $t\neq t_0-T$. where $c_3(n,a,K)=\max\{a+(n-1)K, 0\}$, and
$\kappa=\max\{|\ln D|, 1\}$.
\end{enumerate}
Here, $\alpha:=\max_{\{x|d(x,x_0)=1\}}\Delta_f\,r(x)$,
where $r(x)$ is the distance function to $x$ from base point $x_0$.
\end{theorem}
If $a=0$, theorem recovers the result in \cite{[Wu2]}. We point out
that, similar to Theorem \ref{main}, gradient estimates of Theorem
\ref{main0} also hold provided that only the Bakry-Emery Ricci tensor
is bounded below.

\begin{remark}
In \cite{[Wu2010]} the author proved similar estimates when $m$-Bakry-\'Emery
Ricci tensor is bounded below. He also remarked that $m$-Bakry-\'Emery Ricci
tensor could be replaced by ($\infty$-)Bakry-\'Emery Ricci tensor (see
Remark 1.3 (ii) in \cite{[Wu2010]}). Professor Xiang-Dong Li pointed out
to me that the remark is not accurate because of the lack of global $f$-Laplacian
comparison except some special constraint of $f$ is given. However, Theorem
\ref{main0} corrects my previous remark and provides an answer to this question.
\end{remark}

The rest of this paper is organized as follows. In Section \ref{sec2}, we
will give some auxiliary lemmas and introduce a space-time cut-off function.
These results are prepared to prove Theorem \ref{main} and Theorem \ref{main0}.
In Section \ref{sec3}, we will give completely detail proofs of Theorem
\ref{main} by the classical Yau's gradient estimate technique. Then we will
apply  Theorem \ref{main} to prove Theorem \ref{app1} and Corollary \ref{app1b}.
Meanwhile we will also discuss some special cases of Theorem \ref{main}.
In Section \ref{sec4}, we will adopt the arguments of Theorem 1.1 in \cite{[Wu2]}
to prove Theorem \ref{main0}.

\section{Basic lemmas}\label{sec2}
In this section, we will give some useful lemmas, which are prepared to prove
Theorem \ref{main} and Theorem \ref{main0} in the following sections.
Consider the nonlinear $f$-heat equation
\begin{equation}\label{maineq}
\frac{\partial u}{\partial t}=\Delta_f\, u+au\ln u,
\end{equation}
where $a$ is a real constant, on an $n$-dimensional complete smooth metric measure
space $(M,g,e^{-f}dv)$. For any point $x_0\in M$ and any $R>0$, let
\[
0<u(x,t)\leq D
\]
for some constant $D$, be a smooth solution to \eqref{maineq} in
$Q_{R,T}:\equiv B(x_0,R)\times[t_0-T,t_0]\subset M\times(-\infty,\infty)$,
where $t_0\in \mathbb{R}$ and $T>0$.

\vspace{0.5em}

Similar to \cite{[Lij], [LiZhu]}, we introduce a new smooth function
\[
h(x,t):=u^{1/3}(x,t)
\]
in $Q_{R,T}$. Then
\[
0<h(x,t)\leq D^{1/3}
\]
in $Q_{R,T}$. By \eqref{maineq}, $h(x,t)$ satisfies
\begin{equation}\label{mequk}
\left(\Delta_f-\frac{\partial}{\partial t}\right)h+2h^{-1}|\nabla h|^2+ah\ln h=0.
\end{equation}
Using above, we derive the following evolution formula, which is a generalization
of Lemma 2.1 in \cite{[Jiang]}.
\begin{lemma}\label{Le11}
Let $(M,g,e^{-f}dv)$ be an $n$-dimensional complete smooth metric measure space. For any
point $x_0\in M$ and $R>0$, $Ric_f\geq-(n-1)K$ for some constant $K\geq0$ in $B(x_0,R)$.
Let $0<u(x,t)\leq D$ for some constant $D$, be a smooth solution to \eqref{maineq} in
$Q_{R,T}$. Let
\[
\omega(x,t):=h\cdot|\nabla h|^2,
\]
where $h:=u^{1/3}$. For any $(x,t)\in Q_{R,T}$,
\begin{enumerate}
\item [(i)]if $a\geq0$, then $\omega$ satisfies
\[
\left(\Delta_f-\frac{\partial}{\partial t}\right)\omega\geq -4h^{-1}\left\langle \nabla
h,\nabla\omega\right\rangle+4h^{-3}\omega^2-\left[2(n-1)K+a\ln D+2a\right]\omega.
\]

\item [(ii)]if $a<0$, further assuming that $0<\delta\leq u(x,t)\leq D$ for some constant
$\delta>0$, then $\omega$ satisfies
\[
\left(\Delta_f-\frac{\partial}{\partial t}\right)\omega\geq-4h^{-1}\left\langle \nabla
h,\nabla\omega\right\rangle+4h^{-3}\omega^2-\left[2(n-1)K+a\ln \delta+2a\right]\omega.
\]
\end{enumerate}
\end{lemma}
\begin{proof}
Following the computation method of \cite{[Li-Yau]},
let $e_1, e_2,..., e_n$ be a local orthonormal frame field on $M^n$.
We adopt the notation that subscripts in $i$, $j$, and $k$, with $1\leq i, j, k\leq n$,
mean covariant differentiations in the $e_i$, $e_j$ and $e_k$, directions respectively.

Differentiating $\psi$ in the direction of $e_i$, we have
\begin{equation}\label{emmpro1}
\omega_j=h_j\cdot|\nabla h|^2+2hh_ih_{ij}
\end{equation}
and once more differentiating $\psi$ in the direction of $e_i$,
\[
\Delta\omega=2h|h_{ij}|^2+2hh_ih_{ijj}+4h_ih_jh_{ij}+h^2_ih_{jj},
\]
where $h_i:=\nabla_i h$ and $h_{ijj}:=\nabla_j\nabla_j\nabla_i h$,
etc. Hence we have
\begin{equation*}
\begin{aligned}
\Delta_f\,\omega&=\Delta\omega-\langle\nabla f,\nabla\omega\rangle\\
&=2h|h_{ij}|^2+2hh_ih_{ijj}+4h_ih_jh_{ij}+h^2_ih_{jj}-2hh_{ij}h_if_j-h^2_ih_jf_j\\
&=2h|h_{ij}|^2+2hh_i(\Delta_f h)_i+2hRic_f(\nabla h, \nabla h)+4h_ih_jh_{ij}+h^2_i\Delta_fh.
\end{aligned}
\end{equation*}
By \eqref{mequk}, we also have
\begin{equation*}
\begin{aligned}
\frac{\partial\omega}{\partial t}&=2h\nabla_ih\cdot\nabla_i\left(\Delta_f h+2h^{-1} h_j^2
+ah\ln h\right)+h_t h_j^2\\
&=2h\nabla h \nabla \Delta_f h+8h_ih_jh_{ij}-4h^{-1}h_i^4 +2ah(\ln h+1)h_i^2\\
&\quad+h^2_i\Delta_f h+2h^{-1} |\nabla h|^4 +ah\ln h\cdot h^2_i\\
&=2h\nabla h \nabla \Delta_f h+8h_ih_jh_{ij}-2h^{-1}h_i^4+(3\ln h+2)ahh_i^2+h^2_i\Delta_f h.
\end{aligned}
\end{equation*}
Combining the above two equations, we get
\begin{equation*}
\begin{aligned}
\left(\Delta_f-\frac{\partial}{\partial t}\right)\omega&=2h|h_{ij}|^2
+2h {Ric_f}_{ij} h_ih_j-4h_ih_jh_{ij}\\
&\quad+2h^{-1}h_i^4-(3\ln h+2)ahh_i^2.
\end{aligned}
\end{equation*}
Since $Ric_f\geq-(n-1)K$ for some constant $K\geq0$, then
\begin{equation*}
\begin{aligned}
\left(\Delta_f-\frac{\partial}{\partial t}\right)\omega&\geq2h|h_{ij}|^2+4h_ih_jh_{ij}+2h^{-1}h^4_i-2(n-1)K\,\omega\\
&\quad-8h_ih_jh_{ij}-(3\ln h+2)a\omega.
\end{aligned}
\end{equation*}
Using
\[
2h|h_{ij}|^2+4h_ih_jh_{ij}+2h^{-1}h^4_i\geq 0,
\]
we further get
\begin{equation}\label{dge2}
\left(\Delta_f-\frac{\partial}{\partial t}\right)\omega\geq-8h_ih_jh_{ij}-\left[2(n-1)K+3a\ln h+2a\right]\omega.
\end{equation}
Since \eqref{emmpro1} implies
\[
\omega_jh_j=2hh_ih_jh_{ij}+h_i^4,
\]
using this, \eqref{dge2} can be written by
\[
\left(\Delta_f-\frac{\partial}{\partial t}\right)\omega\geq4h^{-3}\omega^2-4h^{-1}\left\langle \nabla
h,\nabla\omega\right\rangle-\left[2(n-1)K+3a\ln h+2a\right]\omega.
\]
Finally, we notice that if $a\geq 0$, then $0<h\leq D^{1/3}$ and hence
\[
\ln h\leq 1/3\ln D.
\]
If $a<0$, then $\delta^{1/3}\leq h\leq D^{1/3}$ and hence
\[
1/3\ln\delta\leq \ln h\leq 1/3\ln D.
\]
The above two cases imply the desired results.
\end{proof}

\

For equation \eqref{maineq}, if we introduce another new function
\[
g=\ln u,
\]
then $g$ satisfies
\begin{equation}\label{lemequ}
\left(\Delta_f-\frac{\partial}{\partial t}\right)g+|\nabla g|^2+ag=0.
\end{equation}
Using this, we can get the following lemma, which is also a generalization of
previous results in \cite{[Sou-Zh],[Wu2010], [Wu2]}.
\begin{lemma}\label{Lem2}
Let $(M,g,e^{-f}dv)$ be an $n$-dimensional complete smooth metric measure space.
For any point $x_0\in M$ and $R>0$, $Ric_f\geq-(n-1)K$ for some constant
$K\geq0$ in $B(x_0,R)$. Let $0<u(x,t)\leq D$ be a smooth solution to equation
\eqref{maineq} in $Q_{R,T}$. Let $g:=\ln u$ and $\mu:=1+\ln D$. Then for all
$(x,t)\in Q_{R,T}$, the function
\[
\omega:=\left|\nabla\ln(\mu-g)\right|^2=\frac{|\nabla g|^2}{(\mu-g)^2}
\]
satisfies
\begin{equation}
\begin{aligned}\label{lemmaequ3}
\left(\Delta_f-\frac{\partial}{\partial t}\right)\omega&\geq\frac{2(g-\ln D)}{\mu-g}\left\langle \nabla
g,\nabla\omega\right\rangle\\
&\quad+2(\mu-g)\omega^2-2(a+(n-1)K)\omega-\frac{2ag}{\mu-g}\omega.
\end{aligned}
\end{equation}
\end{lemma}
\begin{proof}
The proof of lemma is almost the same as that of \cite{[Wu2010]}. In fact,
in Lemma 2.1 of \cite{[Wu2010]}, if we let $\alpha=1+\ln D$, which means
$\alpha=\mu$, then we have $\delta=1$. Therefore from (2.4) of \cite{[Wu2010]},
we immediately get \eqref{lemmaequ3}.
\end{proof}

\

In the rest of this section, we introduce a smooth cut-off function originated by
Li-Yau \cite{[Li-Yau]} (see also \cite{[BCP]} and \cite{[Wu2]}). This will
also be used in the proof of our theorems.

\begin{lemma}\label{cutoff}
Fix $t_0\in \mathbb{R}$ and $T>0$. For any $\tau\in(t_0-T,t_0]$,
there exists a smooth function
$\bar\psi:[0,\infty)\times[t_0-T,t_0]\to\mathbb R$ such that:
\begin{enumerate}
\item
\[
0\leq\bar\psi(r,t)\leq 1
\]
in $[0,R]\times[t_0-T,t_0]$, and it is supported
in a subset of $[0,R]\times[t_0-T,t_0]$.
\item

\[
\bar\psi(r,t)=1\quad \mathrm{and} \quad\frac{\partial\bar\psi}{\partial r}(r,t)=0
\]
in $[0,R/2]\times[\tau,t_0]$ and $[0,R/2]\times[t_0-T,t_0]$, respectively.
\item
\[
\left|\frac{\partial\bar\psi}{\partial t}\right|\leq\frac{C\bar\psi^{\frac12}}{\tau-(t_0-T)}
\]
in $[0,\infty)\times[t_0-T,t_0]$ for some $C>0$, and $\bar\psi(r,t_0-T)=0$
for all $r\in[0,\infty)$.
\item
\[
-\frac{C_\epsilon\bar\psi^\epsilon}{R}\leq\frac{\partial\bar\psi}{\partial r}\leq 0\quad
\mathrm{and}
\quad \left|\frac{\partial^2\bar\psi}{\partial r^2}\right|\leq\frac{C_\epsilon\bar\psi^\epsilon}{R^2}
\]
in $[0,\infty)\times[t_0-T,t_0]$ for each $\epsilon\in(0,1)$ with some constant
$C_\epsilon$ depending on $\epsilon$.
\end{enumerate}
\end{lemma}
We remind the readers that Lemma \ref{cutoff} is a little different from
that of \cite{[Li-Yau]} and \cite{[Sou-Zh]}. Here, the cut-off function
was previously used by M. Bailesteanua, X. Cao and A. Pulemotov \cite{[BCP]}.

\section{Proof of Theorem \ref{main}}\label{sec3}
In this section,  we will apply Lemmas \ref{Le11} and \ref{cutoff}, the
localization technique of Souplet-Zhang \cite{[Sou-Zh]}, some tricks of
Bailesteanua-Cao-Pulemotov \cite{[BCP]}, Li \cite{[Lij]}, Brighton
\cite{[Bri]} and Jiang \cite{[Jiang]} to prove Theorem \ref{main}.

\begin{proof}[Proof of Theorem \ref{main}]
We only prove the case (i) $a\geq 0$. The case (ii) $a<0$ is similar.
Pick any number $\tau\in(t_0-T,t_0]$ and choose a cutoff function $\bar\psi(r,t)$
satisfying propositions of Lemma \ref{cutoff}. We will show that
\eqref{heor1} holds at the space-time point $(x,\tau)$ for all $x$ such that
$d(x,x_0)<R/2$, where $R\geq 2$. Since $\tau$ is arbitrary,
the conclusion then follows.

Introduce a cutoff function $\psi:M\times[t_0-T,t_0]\to \mathbb R$, such that
\[
\psi=\bar{\psi}(d(x,x_0),t)\equiv\psi(r,t).
\]
Then, $\psi(x,t)$ is supported in $Q_{R,T}$. Our aim is to estimate
$\left(\Delta_f-\frac{\partial}{\partial t}\right)(\psi\omega)$ and
carefully analyze the result at a space-time point where the function
$\psi\omega$ attains its maximum.

By Lemma \ref{Le11}(i), we can calculate that
\begin{equation}
\begin{aligned}\label{lemdx3}
\left(\Delta_f-\frac{\partial}{\partial t}\right)(\psi\omega)&+\left(4h^{-1}\nabla
h-2\frac{\nabla\psi}{\psi}\right)\cdot\nabla(\psi\omega)\\
&\geq 4h^{-3}\omega^2\psi+4h^{-1}\left\langle \nabla
h,\nabla\psi\right\rangle\omega
-2\frac{|\nabla\psi|^2}{\psi}\omega\\
&\quad+(\Delta_f\psi)\omega-\psi_t\omega-\left[2(n-1)K+a\ln D+2a\right]\psi\omega.
\end{aligned}
\end{equation}
Let $(x_1,t_1)$ be a maximum space-time point for
$\psi\omega$ in the closed set
\[
\left\{(x,t)\in M\times[t_0-T,\tau]\,|d(x,x_0)\leq R\right\}.
\]
Assume that $(\psi\omega)(x_1,t_1)>0$; otherwise, $\omega(x,\tau)\leq0$ and
\eqref{heor1} naturally holds at $(x,\tau)$ whenever $d(x, x_0)<\frac R2$.
Here $t_1\neq t_0-T$, since we assume $(\psi\omega)(x_1,t_1)>0$. We can also
assume that function $\psi(x,t)$ is smooth at $(x_1,t_1)$ due to the
standard Calabi's argument \cite{[Cala]}. Since $(x_1,t_1)$ is a maximum
space-time point, at this point,
\[
\Delta_f(\psi\omega)\leq0,\quad(\psi\omega)_t\geq0
\quad \mathrm{and}\quad\nabla(\psi\omega)=0.
\]
Using these, \eqref{lemdx3} at space-time $(x_1,t_1)$
can be simplified as
\begin{equation}\label{lefor}
\begin{aligned}
4\omega^2\psi\leq&\left(-4h^2\left\langle \nabla
h,\nabla\psi\right\rangle+2\frac{|\nabla\psi|^2}{\psi}h^3\right)\omega
-(\Delta_f\psi)h^3\omega+\psi_th^3\omega\\
&+c_1(n,K,a,D)\cdot\psi h^3\omega,
\end{aligned}
\end{equation}
where $c_1(n,K,a,D):=\max\{2(n-1)K+a\ln D+2a,\,0\}$.

\

We apply \eqref{lefor} to prove the theorem. If $x_1\in B(x_0,1)$,
then $\psi$ is constant in space direction in $B(x_0,R/2)$ according
to our assumption, where $R\geq2$. So at $(x_1,t_1)$, \eqref{lefor} yields
\begin{equation*}
\begin{aligned}
\omega&\leq D\left(\frac 14\cdot\frac{\psi_t}{\psi}+\frac{c_1(n,K,a,D)}{4}\right)\\
&\leq D\left(\frac{C}{\tau-(t_0-T)}+\frac{c_1(n,K,a,D)}{4}\right),
\end{aligned}
\end{equation*}
where we used proposition (3) of Lemma \ref{cutoff}. Since $\psi(x,\tau)=1$
when $d(x,x_0)<R/2$ by the proposition (2) of Lemma \ref{cutoff}, the above
estimate indeed gives
\begin{equation*}
\begin{aligned}
\omega(x,\tau)=(\psi\omega)(x,\tau)
&\leq(\psi\omega)(x_1,t_1)\\
&\leq\omega(x_1,t_1)\\
&\leq D\left(\frac{C}{\tau-(t_0-T)}+\frac{c_1(n,K,a,D)}{4}\right)
\end{aligned}
\end{equation*}
for all $x\in M$ such that $d(x,x_0)<R/2$. By the definition of
$w(x,\tau)$ and the fact that $\tau\in(t_0-T,t_0]$ was chosen
arbitrarily, we prove that
\[
\frac{|\nabla u|}{\sqrt{u}}(x,t)\leq \sqrt{D}
\left(\frac{C}{\sqrt{t-t_0+T}}+\frac{1}{2}\sqrt{c_1(n,K,a,D)}
\right)\]
for all $(x,t)\in Q_{R/2,T}$ with $t\neq t_0-T$. This implies \eqref{heor1}.

\vspace{0.5em}

Now, we assume $x_1\not\in B(x_0,1)$. Since
$Ric_f\geq-(n-1)K$ and $r(x_1,x_0)\geq 1$ in $B(x_0,R)$, we have
the $f$-Laplacian comparison (see Theorem 3.1 in \cite{[WW]})
\begin{equation}\label{gencomp}
\Delta_f\,r(x_1)\leq\alpha+(n-1)K(R-1),
\end{equation}
where $\alpha:=\max_{\{x|d(x,x_0)=1\}}\Delta_f\,r(x)$. This comparison
theorem holds without any grow condition of $f$, which is critical in our
latter proof. Below we will estimate upper bounds
for each term of the right-hand side of \eqref{lefor}, similar to the
arguments of Souplet-Zhang \cite{[Sou-Zh]}. Meanwhile, we also
repeatedly use the Young's inequality
\[
a_1a_2\leq \frac{{a_1}^p}{p}+\frac{{a_2}^q}{q},\quad \forall\,\,\, a_1,a_2,p,q>0
\,\,\,\mathrm{with}\,\,\, \frac 1p+\frac 1q=1.
\]
In the following $c$ denotes a constant depending only on $n$ whose value may change
from line to line.

First, we have the estimates of first term of the right hand side
of \eqref{lefor}:
\begin{equation}
\begin{aligned}\label{term1}
-4h^2\left\langle \nabla
h,\nabla\psi\right\rangle\omega
&\leq 4h^{3/2}\cdot|\nabla\psi|\cdot\omega^{3/2}\\
&\leq 4D^{1/2}\cdot|\nabla\psi|\psi^{-3/4}\cdot(\psi\omega^2)^{3/4}\\
&\leq\frac 35\psi\omega^2+c D^2\frac{|\nabla\psi|^4}{\psi^3}\\
&\leq\frac 35\psi\omega^2+c\frac{D^2}{R^4}.
\end{aligned}
\end{equation}
For the second term of the right hand side of (\ref{lefor}), we have
\begin{equation}
\begin{aligned}\label{term2}
2\frac{|\nabla\psi|^2}{\psi}h^3 \omega
&\leq2D\cdot|\nabla\psi|^2\psi^{-3/2}\cdot\psi^{1/2}\omega \\
&\leq\frac 35\psi\omega^2+c D^2\frac{|\nabla\psi|^4}{\psi^3}\\
&\leq\frac 35\psi\omega^2+c\frac{D^2}{R^4}.
\end{aligned}
\end{equation}
For the third term of the right hand side of \eqref{lefor}, since $\psi$
is a radial function, then at $(x_1,t_1)$, using \eqref{gencomp} we have
\begin{equation}
\begin{aligned}\label{term3}
-(\Delta_f\psi)h^3\omega&=-\left[(\partial_r\psi)\Delta_fr+(\partial^2_r\psi)\cdot
|\nabla r|^2\right]h^3\omega\\
&\leq-\left[\partial_r\psi\left(\alpha+(n-1)K(R-1)\right)
+\partial^2_r\psi\right]h^3\omega\\
&\leq D\left[|\partial^2_r\psi|+\left(|\alpha|+(n-1)K(R-1)\right)|\partial_r\psi|\right]\omega\\
&=D\psi^{1/2}\omega\frac{|\partial^2_r\psi|}{\psi^{1/2}}
+D\left(|\alpha|+(n-1)K(R-1)\right)\psi^{1/2}\omega
\frac{|\partial_r\psi|}{\psi^{1/2}}\\
&\leq\frac 35\psi\omega^2{+}c\,D^2
\left[\left(\frac{|\partial^2_r\psi|}{\psi^{1/2}}\right)^2
{+}\left(\frac{|\alpha|\cdot|\partial_r\psi|}{\psi^{1/2}}\right)^2
{+}\left(\frac{K(R{-}1)|\partial_r\psi|}{\psi^{1/2}}\right)^2\right]\\
&\leq\frac 35\psi\omega^2+c\frac{D^2}{R^4}+c\frac{\alpha^2D^2}{R^2}
+cK^2D^2,
\end{aligned}
\end{equation}
where in the last inequality we used proposition (4) of Lemma \ref{cutoff}.

Then we estimate the fourth term of the right hand side of \eqref{lefor}:
\begin{equation}
\begin{aligned}\label{term4}
|\psi_t|h^3\omega&=\psi^{1/2}\omega\frac{h^3|\psi_t|}{\psi^{1/2}}\\
&\leq\frac 35\left(\psi^{1/2}\omega\right)^2+c
\left(\frac{h^3|\psi_t|}{\psi^{1/2}}\right)^2\\
&\leq\frac 35\psi\omega^2+\frac{cD^2}{(\tau-t_0+T)^2}.
\end{aligned}
\end{equation}

Finally, we estimate the last term of the right hand side of \eqref{lefor}:
\begin{equation}\label{term5}
c_1(n,K,a,D)\psi h^3\omega\leq\frac 35\psi\omega^2+cD^2c_1^2(n,K,a,D).
\end{equation}

\

We now substitute \eqref{term1}-\eqref{term5} into the right hand side of \eqref{lefor},
and get that
\begin{equation}\label{leforfor}
\psi\omega^2\leq c\,D^2\left(\frac{1}{R^4}{+}\frac{\alpha^2}{R^2}+\frac{1}{(\tau-t_0+T)^2}+K^2+c_1^2(n,K,a,D)\right)
\end{equation}
at $(x_1,t_1)$. This implies that
\begin{equation*}
\begin{aligned}
(\psi^2\omega^2)(x_1,t_1)
&\leq(\psi\omega^2)(x_1,t_1)\\
&\leq c D^2\left(\frac{1}{R^4}{+}\frac{\alpha^2}{R^2}+\frac{1}{(\tau-t_0+T)^2}+K^2+c_1^2(n,K,a,D)\right).
\end{aligned}
\end{equation*}
Since $\psi(x,\tau)=1$ when $d(x,x_0)<R/2$ by the proposition (2)
of Lemma \ref{cutoff}, from the above estimate, we have
\begin{equation*}
\begin{aligned}
\omega(x,\tau)&=(\psi\omega)(x,\tau)\\
&\leq(\psi\omega)(x_1,t_1)\\
&\leq c\,D\left(\frac{1}{R^2}{+}\frac{|\alpha|}{R}+\frac{1}{\tau-t_0+T}+K+c_1(n,K,a,D)\right)
\end{aligned}
\end{equation*}
for all $x\in M$ such that $d(x,x_0)<R/2$. By the definition of
$w(x,\tau)$ and the fact that $\tau\in(t_0-T,t_0]$ was chosen
arbitrarily, we in fact show that
\[
\sqrt{h}|\nabla h|(x,t)\leq c\sqrt{D}\left(\frac{1}{R}{+}\frac{\sqrt{|\alpha|}}{\sqrt{R}}{+}\frac{1}{\sqrt{t{-}t_0{+}T}}
{+}\sqrt{K}{+}\sqrt{c_1(n,K,a,D)}\right)
\]
for all $(x,t)\in Q_{R/2,T}\equiv B(x_0,R/2)\times[t_0-T,t_0]$ with
$t\neq t_0-T$. We have finished the proof of theorem since $h=u^{1/3}$
and $R\geq2$.
\end{proof}

\

In particular, if $a=0$, Theorem \ref{main} implies a local elliptic gradient
estimate for the $f$-heat equation:
\begin{equation}\label{cheorf}
\frac{|\nabla u|}{\sqrt{u}}\leq c(n)\, \sqrt{D}\left(\frac{1}{R}+\sqrt{\frac{|\alpha|}{R}}
+\frac{1}{\sqrt{t-t_0+T}}+\sqrt{K}\right)
\end{equation}
in $Q_{R/2,T}$ with $t\neq t_0-T$, for any $R\geq2$. Compared with author's
recent result \cite{[Wu2]}, though \eqref{cheorf} is not sharp from
Example 1.2 of \cite{[Wu2]}, it seems to be a new form of elliptic
type gradient estimates for $f$-heat equation.

\

Furthermore, if $a=0$ and $f$ is constant, by using the classical Laplacian
comparison $\Delta r\leq(n-1)(1/r+\sqrt{K})$ instead of Wei-Wylie's $f$-Laplacian
comparison (see \eqref{gencomp}), the proof of Theorem \ref{main} in fact implies
the following gradient estimate for the heat equation:
\begin{equation}\label{cheorf2}
\frac{|\nabla u|}{\sqrt{u}}\leq c(n)\, \sqrt{D}\left(\frac{1}{R}
+\frac{1}{\sqrt{t-t_0+T}}+\sqrt{K}\right)
\end{equation}
in $Q_{R/2,T}$ with $t\neq t_0-T$, for all $R>0$. Compared with Hamilton's
estimate \cite{[Ham93]} and Souplet-Zhang's estimate \cite{[Sou-Zh]} for
the heat equation, this elliptic gradient estimate seems to be new.

\

Moreover, gradient estimate \eqref{cheorf} implies
\begin{corollary}\label{app2}
Let $(M,g,e^{-f}dv)$ be an $n$-dimensional complete smooth metric
measure space with $Ric_f\geq -(n-1)K$ for some constant $K\geq 0$.
If $u(x,t)$ is a positive ancient solution to the $f$-heat equation
\eqref{weiheat} such that $0<u(x,t)\leq D$ for some constant $D$,
then $|\nabla u|\leq c(n)D\sqrt{K}$.
\end{corollary}

\begin{remark}
Corollary \ref{app2} implies Brighton's result \cite{[Bri]}: any positive
bounded $f$-harmonic function on complete noncompact smooth metric measure
spaces with $Ric_f\geq 0$ must be constant.
\end{remark}

\

In the rest of this section, we shall apply Theorem \ref{main} to prove
Theorem \ref{app1}, Corollary \ref{app1b} and Corollary \ref{app2}.
\begin{proof}[Proof of Theorem \ref{app1}]
When $a>0$, since $K=0$ and $D=e^{-2}$, then $c_1(n,K,a,D)=0$. Fixing
any space-time point $(x_0,t_0)$ and using Theorem \ref{main} for
$0<u\leq e^{-2}$ in the set $B(x_0,R)\times(t_0-R^2, t_0]$, we have
\[
\frac{|\nabla u|}{\sqrt{u}}(x_0,t_0)\leq c(n)e^{-2}\left(\sqrt{\frac{1+|\alpha|}{R}}
+\frac{1}{R}\right)
\]
for all $R\geq2$. Letting $R\to\infty$, then
\[
|\nabla u(x_0, t_0)|=0.
\]
Since $(x_0, t_0)$ is arbitrary, $u$ must be constant in $x$. By
equation \eqref{equ1}, then
\[
\frac{d u}{d t}=a\,u\ln u.
\]
Solving this ODE, we get
\[
u=\exp(ce^{at}),
\]
where $c$ is some constant and $a>0$. For such a solution, if we
let $t\to -\infty$, then
\[
u=\exp(ce^{at})\to 1,
\]
which is contradiction with the theorem assumption: $0<u(x,t)\leq e^{-2}$.
Therefore such $u$ does not exist.

When $a<0$, since $K=0$ and $\delta=e^{-2}$, then $c_2(n,K,a,D)=0$.
For any space-time point $(x_0,t_0)$, we apply Theorem \ref{main} for
$e^{-2}\leq u(x,t)\leq D$ in the set $B(x_0,R)\times(t_0-R^2, t_0]$,
\[
\frac{|\nabla u|}{\sqrt{u}}(x_0,t_0)\leq c(n)\sqrt{D}\left(\sqrt{\frac{1+|\alpha|}{R}}
+\frac{1}{R}\right)
\]
for all $R\geq2$. Similar to the above arguments, letting $R\to\infty$, then $u$ is
constant in $x$, and $u=\exp(ce^{at})$ for some constant $c$.
When $t\to -\infty$, we observe that:
$u=\exp(ce^{at})\to +\infty$ if $c>0$;
$u=\exp(ce^{at})\to 0$ if $c<0$;  $u=1$ if $c=0$. Moreover, the theorem
assumption requires $e^{-2}\leq u(x,t)\leq D$. Hence $u$ only exists when
$D\geq 1$ and the desired result follows.

When $a=0$, $K=0$, and assume that $u(x,t)$ is a positive ancient solution to
equation \eqref{weiheat} such that $u(x,t)=o\Big(\big[r^{1/2}(x)+|t|^{1/4}\big]^2\Big)$
near infinity. Fixing any space-time $(x_0,t_0)$ and using \eqref{cheorf} for $u$ on the set
$B(x_0,R)\times(t_0-R^2, t_0]$, we obtain
\[
\frac{|\nabla u|}{\sqrt{u}}(x_0,t_0)\leq c(n) \left(\sqrt{\frac{1+|\alpha|}{R}}
+\frac{1}{R}\right)\cdot o(\sqrt{R})
\]
for all $R\geq2$.  Letting $R\to\infty$, it follows that
\[
|\nabla u(x_0, t_0)|=0.
\]
Since $(x_0, t_0)$ is arbitrary, we get $u$ is constant in space-time.
\end{proof}

Theorem \ref{app1} in fact implies Corollary \ref{app1b}.
\begin{proof}[Proof of Corollary \ref{app1b}]
If positive smooth function $u(x)$ achieves the weighted log-Sobolev
constant $S_M$ and satisfies $\int_M u^2 e^{-f}dv=V_f(M)$, then from
the introduction above, $u(x)$ satisfies elliptic equation \eqref{sobl}.
Assume that our conclusion is incorrect, that is, $0<u\leq e^{-2}$.
Since $S_M>0$, by Theorem \ref{app1} Case (i), there does not exist
such function $u$ satisfying \eqref{sobl}. This is a contradiction.
\end{proof}

\begin{proof}[Proof of Corollary \ref{app2}]
When $a=0$, for any space-time point $(x_0,t_0)$, using estimate
\eqref{cheorf} for $0<u\leq D$ in the set $B(x_0,R)\times(t_0-R^2, t_0]$,
\[
\frac{|\nabla u|}{\sqrt{u}}(x_0,t_0)\leq c(n)\sqrt{D}\left(\sqrt{\frac{1+|\alpha|}{R}}
+\frac{1}{R}+\sqrt{K}\right)
\]
for all $R\geq2$. Letting $R\to\infty$, then
\[
|\nabla u(x_0, t_0)|\leq c(n)D\sqrt{K}.
\]
Since $(x_0, t_0)$ is arbitrary, the result follows.
\end{proof}

\section{Proof of Theorem \ref{main0}}\label{sec4}
In this section, we will prove Theorem \ref{main0}. The proof is analogous to
Theorem 1.1 in \cite{[Wu2]}. For the readers convenience, we provide a detailed
proof. Compared with the previous proof, here we need to carefully deal with an extra
nonlinear term.

\begin{proof}[Proof of Theorem \ref{main0}]
We only consider the case $a\geq 0$. The case $a<0$ is similar.
Using Lemma \ref{Lem2}, we calculate that
\begin{equation}
\begin{aligned}\label{lembud2}
&\Delta_f(\psi\omega)-\frac{2(g-\ln D)}{\mu-g}\nabla
g\cdot\nabla(\psi\omega) -2\frac{\nabla\psi}{\psi}
\cdot\nabla(\psi\omega)-(\psi\omega)_t\\
\geq&\,2\psi(\mu-g)\omega^2
-\left[\frac{2(g-\ln D)}{\mu-g}\nabla g\cdot\nabla\psi\right]
\omega-2\frac{|\nabla\psi|^2}{\psi}\omega\\
&+(\Delta_f\psi)\omega-\psi_t\omega
-2(a+(n-1)K)\psi\omega-\frac{2a\,g}{\mu-g}\psi\omega.
\end{aligned}
\end{equation}
Let $(x_1,t_1)$ be a point where $\psi\omega$  achieves the maximum.

We first consider the case $x_1\not\in B(x_0,1)$.
By Li-Yau \cite{[Li-Yau]}, without loss of generality we assume that
$x_1$ is not in the cut-locus of $M$. Then at this point, we have
\begin{equation*}
\begin{aligned}
\Delta_f(\psi\omega)\leq0,\,\,\,\,\,\,(\psi\omega)_t\geq0,
\,\,\,\,\,\,\nabla(\psi\omega)=0.
\end{aligned}
\end{equation*}
Hence by \eqref{lembud2}, at $(x_1,t_1)$, we get
\begin{equation}
\begin{aligned}\label{leforkk}
2\psi(\mu-g)\omega^2&\leq
\Bigg\{\left(\frac{2(g-\ln D)}{\mu-g}\nabla g\cdot\nabla\psi\right)\omega
+2\frac{|\nabla\psi|^2}{\psi}\omega-(\Delta_f\psi)\omega\\
&\quad\,\,\,+\psi_t\omega+2(a+(n-1)K)\psi\omega
+\frac{2a\,g}{\mu-g}\psi\omega\Bigg\}.
\end{aligned}
\end{equation}
We will carefully estimate the upper bounds for each term
of the right-hand side of \eqref{leforkk}. Similar to arguments of
Section \ref{sec3}, we still repeatedly use the Young's inequality.
For the first term of right hand side of \eqref{leforkk}, we have
\begin{equation}\label{sjgj1}
\left(\frac{2(g-\ln D)}{\mu-g}\nabla
g\cdot\nabla\psi\right)\omega\leq(\mu-g)\psi \omega^2+c\frac{|\ln D-g|}{R^4}.
\end{equation}
For the second term of the right hand side of \eqref{leforkk}, we get
\begin{equation}\label{sjgj2}
2\frac{|\nabla\psi|^2}{\psi}\omega
\leq\frac{1}{10}\psi\omega^2+\frac{c}{R^4}.
\end{equation}
For the third term of the right hand side of \eqref{leforkk}, we have
\begin{equation}
\begin{aligned}\label{sjgj3}
-(\Delta_f\psi)\omega&=-\left[(\partial_r\psi)\Delta_fr+(\partial^2_r\psi)\cdot
|\nabla r|^2\right]\omega\\
&\leq-\left[\partial_r\psi\left(\alpha+(n-1)K(R-1)\right)
+\partial^2_r\psi\right]\omega\\
&\leq \left[|\partial^2_r\psi|+\left(|\alpha|+(n-1)K(R-1)\right)|\partial_r\psi|\right]\omega\\
&=\psi^{1/2}\omega\frac{|\partial^2_r\psi|}{\psi^{1/2}}
+\left(|\alpha|+(n-1)K(R-1)\right)\psi^{1/2}\omega
\frac{|\partial_r\psi|}{\psi^{1/2}}\\
&\leq\frac{\psi\omega^2}{10}+c
\left[\left(\frac{|\partial^2_r\psi|}{\psi^{1/2}}\right)^2
+\left(\frac{|\alpha|\cdot|\partial_r\psi|}{\psi^{1/2}}\right)^2
+\left(\frac{K(R-1)|\partial_r\psi|}{\psi^{1/2}}\right)^2\right]\\
&\leq\frac{1}{10}\psi\omega^2+\frac{c}{R^4}+c\frac{\alpha^2}{R^2}
+cK^2,
\end{aligned}
\end{equation}
where the $f$-Laplacian comparison was used. Here, since
$Ric_f\geq-(n-1)K$ and $r(x_1,x_0)\geq 1$ in $B(x_0,R)$, we have
the $f$-Laplacian comparison (see Theorem 3.1 in \cite{[WW]})
\[
\Delta_f\,r(x_1)\leq\alpha+(n-1)K(R-1),
\]
where $\alpha=\max_{\{x|d(x,x_0)=1\}}\Delta_f\,r(x)$.
For the fourth term of the right hand side of \eqref{leforkk}, we have
\begin{equation}
\begin{aligned}\label{sjgj4}
|\psi_t|\omega&=\psi^{1/2}\omega\frac{|\psi_t|}{\psi^{1/2}}\\
&\leq\frac{1}{10}\left(\psi^{1/2}\omega\right)^2+c
\left(\frac{|\psi_t|}{\psi^{1/2}}\right)^2\\
&\leq\frac{1}{10}\psi\omega^2+\frac{c}{(\tau-t_0+T)^2}.
\end{aligned}
\end{equation}
For the fifth term of the right hand side of \eqref{leforkk}, we have
\begin{equation}\label{sjgj5}
2(a+(n-1)K)\psi\omega\leq\frac{1}{10}\psi\omega^2+c(a+(n-1)K)^2.
\end{equation}
For the sixth term of the right hand side of \eqref{leforkk}, we have
\begin{equation}\label{sjgj6}
\frac{2a\,g}{\mu-g}\psi\omega\leq
\frac{1}{10}\psi\omega^2+\frac{c a^2g^2}{(\mu-g)^2}.
\end{equation}
Now at $(x_1,t_1)$, we substitute \eqref{sjgj1}-\eqref{sjgj6} to the right
hand side of \eqref{leforkk} and obtain
\begin{equation}
\begin{aligned}\label{leforfor}
2\psi(\mu-g)\omega^2&\leq \psi(\mu-g)\omega^2+c\frac{|\ln D-g|}{R^4}
+\frac{\psi\omega^2}{2}+\frac{c}{(\tau-t_0+T)^2}\\
&\,\,\,\,\,\,+\frac{c}{R^4}+c\frac{\alpha^2}{R^2}
+cK^2+c(a+(n-1)K)^2+\frac{ca^2g^2}{(\mu-g)^2}.
\end{aligned}
\end{equation}
Recalling that
\[
\mu-g \geq1\quad \mathrm{and}\quad \frac{|\ln D-g|}{\mu-g}\leq 1,
\]
then \eqref{leforfor} implies
\begin{equation}
\begin{aligned}\label{lefogong}
\psi \omega^2&\leq \frac{c}{R^4}+\frac{c}{(\tau-t_0+T)^2}\\
&\,\,\,\,\,\, +c\frac{\alpha^2}{R^2}
+cK^2+c(a+(n-1)K)^2+\frac{ca^2g^2}{(\mu-g)^2}
\end{aligned}
\end{equation}
at space-time $(x_1,t_1)$. By some basic analysis, we \emph{claim} that:
\begin{equation}\label{baseest}
\frac{g^2}{(\mu-g)^2}\leq \kappa^2, \quad\mathrm{where}\quad \kappa:=\max\{|\ln D|, 1\},
\end{equation}
for all $g\leq \ln D$, where the constant $\mu:=1+\ln D$. To see this, notice that function
$\frac{g^2}{(\mu-g)^2}$ has only one critical point $g=0$, and it is continuous
on $(-\infty, \ln D]$ satisfying
\[
\lim_{g\to-\infty}\frac{g^2}{(\mu-g)^2}=1\quad
\mathrm{and}\quad\lim_{g\to \ln D}\frac{g^2}{(\mu-g)^2}=|\ln D|^2.
\]
Hence \eqref{baseest} easily follows.

Using \eqref{baseest}, inequality \eqref{lefogong} becomes
\[
(\psi \omega^2)(x_1,t_1)\leq c\left(\frac{\alpha^2+1}{R^2}+\frac{1}{(\tau-t_0+T)^2}
+K^2+(a+(n-1)K)^2+a^2\kappa^2\right),
\]
where we used $R\geq2$. This implies that
\begin{equation*}
\begin{aligned}
(\psi^2\omega^2)(x_1,t_1)&
\leq(\psi\omega^2)(x_1,t_1)\\
&\leq c\left(\frac{\alpha^2+1}{R^2}+\frac{1}{(\tau-t_0+T)^2}
+K^2+(a+(n-1)K)^2+a^2\kappa^2 \right).
\end{aligned}
\end{equation*}
Since $\psi(x,\tau)=1$ when $d(x,x_0)<R/2$ by the proposition (2)
of Lemma \ref{cutoff}, from the above estimate, we have
\begin{equation*}
\begin{aligned}
\omega(x,\tau)&=(\psi\omega)(x,\tau)\\
&\leq(\psi\omega)(x_1,t_1)\\
&\leq c\left(\frac{|\alpha|+1}{R}+\frac{1}{\tau-t_0+T}
+K+a+(n-1)K+a\,\kappa\right)
\end{aligned}
\end{equation*}
for all $x\in M$ such that $d(x,x_0)<R/2$. By the definition of
$w(x,\tau)$ and the fact that $\tau\in(t_0-T,t_0]$ was chosen
arbitrarily, we in fact show that
\[
\frac{|\nabla
g|}{\mu-g}(x,t)\leq c\left(\sqrt{\frac{1{+}|\alpha|}{R}}+\frac{1}{\sqrt{t-t_0+T}}{+}\sqrt{K}+\sqrt{a(\kappa+1)}\right)
\]
for all $(x,t)\in Q_{R/2,T}\equiv B(x_0,R/2)\times[t_0-T,t_0]$ with
$t\neq t_0-T$. Since $g=\ln u$ and $\mu=1+\ln D$, the theorem follows when $x_1\not \in B(x_0,1)$.

\vspace{0.5em}

Now we consider the other case: $x_1\in B(x_0,1)$. In this case, $\psi$ is constant in space
direction in $B(x_0,R/2)$ by our assumption, where $R\geq2$. So
at $(x_1,t_1)$, \eqref{leforkk} yields
\begin{equation*}
\begin{aligned}
2(\mu-g)\omega&\leq
\frac{\psi_t}{\psi}+2(a+(n-1)K)+\frac{2a\,g}{\mu-g}\\
&\leq\frac{C}{\tau-(t_0-T)}+2(a+(n-1)K)+2a\kappa,
\end{aligned}
\end{equation*}
where we used proposition (3) of Lemma \ref{cutoff}. Since
$\mu-g\geq 1$ and $\psi(x,\tau)=1$
when $d(x,x_0)<R/2$ by the proposition (2) of Lemma \ref{cutoff}, the above
estimate indeed gives
\begin{equation*}
\begin{aligned}
\omega(x,\tau)&=(\psi\omega)(x,\tau)\\
&\leq(\psi\omega)(x_1,t_1)\\
&\leq\omega(x_1,t_1)\\
&\leq \frac{C}{\tau-(t_0-T)}+(a+(n-1)K)+a\kappa
\end{aligned}
\end{equation*}
for all $x\in M$ such that $d(x,x_0)<R/2$. By the definition of
$w(x,\tau)$ and the fact that $\tau\in(t_0-T,t_0]$ was chosen
arbitrarily, we in fact prove that
\[
\frac{|\nabla
g|}{\mu-g}(x,t)\leq \frac{C}{\sqrt{t-t_0+T}}+\sqrt{(n-1)K}+\sqrt{a(\kappa+1)}
\]
for all $(x,t)\in Q_{R/2,T}$ with $t\neq t_0-T$. So \eqref{heor1} is still true.
\end{proof}


\bibliographystyle{amsplain}

\end{document}